\documentclass[reqno,12pt]{amsart}
\usepackage[utf8]{inputenc}
\usepackage[T1]{fontenc}
\usepackage[english]{babel}
\usepackage[babel]{microtype}
\usepackage[headings]{fullpage}
\usepackage{amsmath,amssymb,amsthm,amsrefs}
\usepackage{mathrsfs,mathtools}
\usepackage{bbm}
\usepackage{dsfont}
\usepackage{lmodern}
\usepackage{cases}
\usepackage{enumitem}
\usepackage{centernot}
\usepackage{hyperref}

\hypersetup{colorlinks=true,
  linkcolor=blue,
  citecolor=blue,
  pdfpagemode=UseNone,
  pdfstartview={XYZ null null 1.00}
}

\DeclareFontShape{OMX}{cmex}{m}{n}{%
  <-7.5> cmex7
  <7.5-8.5> cmex8
  <8.5-9.5> cmex9
  <9.5-> cmex10
}{}
\DeclareSymbolFont{largesymbols}{OMX}{cmex}{m}{n} 

\newcommand{\from}{\colon}

\newcommand{\divp}[1]{\mathcal{D}_{#1}}
\newcommand{\divn}{\divp{[n]}}
\newcommand{\cube}[2]{\mathcal{Q}^{#1}_{#2}}
\newcommand{\multi}[2]{\mathcal{M}^{#1}_{#2}}

\newcommand{\embeds}{\hookrightarrow}

\newcommand{\pdim}[1]{\dim\mathopen{}\big(#1\big)\mathclose{}}
\newcommand{\twodim}[1]{\dim_2 \mathopen{}\big(#1\big)\mathclose{}}

\newcommand{\multiset}[2]{\ensuremath{\left(\kern-.3em\left(\genfrac{}{}{0pt}{}{#1}{#2}\right)\kern-.3em\right)}}

\numberwithin{equation}{section}

\newtheorem{theorem}    {Theorem}[section]
\newtheorem{lemma}      [theorem] {Lemma}
\newtheorem{conjecture} [theorem] {Conjecture}

\newtheorem{corollary}  [theorem] {Corollary}

\newtheorem{question}   [theorem] {Question}

\newtheorem{problem}    [theorem] {Problem}

\DeclareMathOperator{\Z}{\mathbb{Z}}
\DeclareMathOperator{\N}{\mathbb{N}}
\DeclareMathOperator{\R}{\mathbb{R}}

\DeclareMathOperator{\Prob}{\mathbb{P}}

\begin{document}

\author{David Lewis}
\author{Victor Souza}
\title{The order dimension of divisibility}

\address{Department of Mathematical Sciences,
University of Memphis, Memphis, TN 38152, USA}
\email{davidcharleslewis@outlook.com}

\address{IMPA, Estrada Dona Castorina 110,
Jardim Bot\^anico, Rio de Janeiro, 22460-320, Brazil}
\email{souza@impa.br}


\begin{abstract}
The Dushnik-Miller dimension of a partially-ordered set $P$
is the smallest $d$ such that one can embed $P$ into
a product of $d$ linear orders.
We prove that the dimension of the divisibility order
on the interval $\{1, \dotsc, n\}$, is equal to
${(\log n)^2}(\log\log n)^{-\Theta(1)}$
as $n$ goes to infinity.

We prove similar bounds for the $2$-dimension
of divisibility in $\{1, \dotsc, n\}$,
where the $2$-dimension of a poset $P$
is the smallest $d$ such that $P$ is isomorphic to
a suborder of the subset lattice of $[d]$.
We also prove an upper bound for the $2$-dimension
of posets of bounded degree
and show that the $2$-dimension of the divisibility poset
on the set $(\alpha n, n]$ is $\Theta_\alpha(\log n)$
for $\alpha \in (0,1)$.
At the end we pose several problems.
\end{abstract}

\clearpage\maketitle
\thispagestyle{empty}


\section{Introduction}
\label{sec:intro}

The Dushnik-Miller \emph{dimension} (hereafter, dimension)
of a poset is a fundamental concept in the study of partial orders.
First introduced by Dushnik and Miller~\cite{dm} in 1941,
$\pdim{P}$ is defined as the minimum $d$ such that the poset $P$
can be embedded into a product of $d$ linear orders.

For any subset $S \subseteq \N$,
denote by $\divp{S}$ the divisibility poset restricted to the set $S$.
Properties of the divisibility order have been studied,
for example, by Cameron and Erd\H{o}s~\cite{ce}.
Surprisingly, the dimension of the divisibility order,
as far as we know, has not been considered in the literature.
Since the dimension of $\divp{\N}$ is infinite,
we are usually concerned with the case where $S$ is finite.
Indeed, we are primarily interested in the case
$S = [n] := \{ 1, \dotsc, n\}$.
In our main result,
we determine the growth of $\pdim{\divn}$ as $n$ goes to infinity
up to a $\log \log n$ factor.

\begin{theorem}
\label{thm:main}
The dimension of $\divn$,
the divisibility order on $[n]$,
satisfies, as $n \to \infty$,
\begin{equation*}
  \big( \tfrac{1}{16} - o(1) \big) \frac{(\log n)^2}{(\log \log n)^2}
  \leq \pdim{\divn}
  \leq \big( 4 + o(1) \big) \frac{(\log n)^2}{\log \log n}.
\end{equation*}
\end{theorem}

Note that we always write $\log$ for the natural logarithm.
Unlike with other natural suborders of $\divp{\N}$,
such as the set of divisors of a given natural number,
the dimension of $\divn$ doesn't seem to reduce
to a well-known number-theoretic function.
For example, the poset of divisors of $n$
(which is the interval from $1$ to $n$ with respect to the divisibility order)
is just a product of $\omega(n)$ chains
and so has dimension $\omega(n)$,
where $\omega(n)$ is the number of distinct prime factors of $n$.
But the set $[n]$ is an interval in the usual order on the integers,
and it displays a nontrivial interaction
with the divisibility order when regarding the dimension.

We prove Theorem~\ref{thm:main}
by embedding a suborder of the hypercube into $\divn$,
then embedding $\divn$ into a product of simple posets
and showing that each of them has small dimension.
We observe that this same idea works,
with small modifications,
in a variety of circumstances.
For example, $t$-dimension,
where $t$ is an integer greater than or equal to $2$,
is a variant of dimension introduced by Nov\'ak~\cite{novak}.
We are most interested in the case $t=2$.
The $2$-dimension of a poset $P$,
denoted $\twodim{P}$,
is the smallest $d$ such that there is
an embedding from $P$ into the hypercube $\cube{d}{}$,
the poset of subsets of $[n]$ ordered by inclusion.
We prove an analogue of Theorem~\ref{thm:main} for $2$-dimension as well.

\begin{theorem}
\label{thm:main2}
The $2$-dimension of $\divn$ satisfies,
as $n \to \infty$,
\begin{equation*}
  \big(\tfrac{1}{16} - o(1) \big) \frac{(\log n)^2}{(\log \log n)^2}
  \leq \twodim{\divn}
  \leq \big(\tfrac{4}{3}e \pi^2 + o(1) \big)\frac{(\log n)^2}{\log\log n}.
\end{equation*}
\end{theorem}

We also consider other natural choices
for subsets of $\N$ to bound the dimension.
Some sets like $(n^\alpha, n]$
and $a[n] + b = \{ak + b : k \in [n] \}$
behave similarly to $[n]$ with respect to the dimension.
On the other hand,
in Section~\ref{sec:alpha} we shall see that the dimension
of divisibility over the set $(\alpha n , n]$
behaves quite differently.
In fact, if $\alpha \geq 1/2$,
then $\divp{(\alpha n, n]}$ is an antichain
and thus has dimension $2$.
Using a result of Scott and Wood \cite{scottwood}
on posets with bounded degree,
we show that $\divp{(\alpha n, n]}$ has bounded dimension,
and that, as $\alpha \to 0$,
\begin{equation*}
\sup_{n \in \N} \pdim{\divp{(\alpha n, n]}}
\leq \tfrac{1}{\alpha} \big(\log(\tfrac{1}{\alpha})\big)^{1+o(1)}.
\end{equation*}
In Section~\ref{sec:alpha},
we prove an analogue of a result by F\"uredi and Kahn~\cite{furedikahn}
for $2$-dimension and use it to show that
$\twodim{\divp{(\alpha n, n]}} = \Theta_\alpha(\log n)$
as $n \to \infty$,
and that the same holds for $t$-dimension for any $t\geq 2$.

While combinatorial properties of the divisibility poset
have been studied before,
results often are not stated in the language of partial orders.
For example, Cameron and Erd\H{o}s~\cite{ce} defined \emph{primitive sets}
of integers as antichains in the divisibility order,
and conjectured that the number of primitive subsets of $[n]$ is
$(\alpha+o(1))^n$ for some constant $\alpha$.
This conjecture was recently proven by Angelo~\cite{angelo}.
Continuing this work, Liu, Pach, and Palincza~\cite{lpp}
proved that the number of maximum-size primitive subsets
of $[n]$ is $(\beta+o(1))^n$ for some constant $\beta$,
and gave algorithms for computing
both $\alpha \approx 1.57$ and $\beta \approx \sqrt{1.318}$.
They also showed that the number of \emph{strong antichains}
in $\divp{[2,n]}$ is $2^{\pi(n)}\cdot e^{(1+o(1))\sqrt{n}}$,
where a strong antichain in a poset $P$ is a subset of $P$
such that no two elements have a common lower bound in $P$,
and $\pi(n)$ is the number of primes up to $n$.
We hope this note motivates further work on
the combinatorial aspects of the divisibility order.


\section{Dimension of Posets}
\label{sec:dimension}

A \emph{poset} is an ordered pair $P = (S, \leq_P)$,
where $S$ is a set and $\leq_P$ is a partial order.
We usually identify a poset with its ground set,
especially when it's clear which partial order we're using.
For $a,b \in S$, we write $a <_P b$ to mean
that $a \leq_P b$ and $a \neq b$.
Two elements $a, b \in S$ are \emph{incomparable}
if neither $a \leq_P b$ nor $b \leq_P a$ hold.
A \emph{linear order} is a poset in which the elements are pairwise comparable.
Given two posets $P = (S, \leq_P)$, $Q = (S', \leq_Q)$,
a \emph{poset embedding} from $P$ into $Q$
is a map $\varphi: S \to S'$
such that $\varphi(a) \leq_Q \varphi(b)$
if and only if $a \leq_P b$.
The expression $P\embeds Q$
represents an embedding from $P$ into $Q$
or the existence of such an embedding, depending on context.
A poset $Q = (T, \leq_Q)$ is called a \emph{suborder}
of $P = (S, \leq_P)$ if $T\subseteq S$
and the inclusion map is a poset embedding.
A \emph{chain} in a poset is a suborder that is a linear order,
and an \emph{antichain} is a suborder in which
the elements are pairwise incomparable.

Given a family of posets $P_i = (S_i, \leq_{P_i})$, $i \in I$,
the \emph{product poset} $P = \prod_{i \in I} P_i$
is the unique order on the product set $S = \prod_{i \in I} S_i$
such that $a \leq_P b$
if and only if $a_i \leq_{P_i} b_i$ for all $i \in I$.
Consider $\R$ with its standard order.
The \emph{dimension} of a countable poset $P$,
denoted $\pdim{P}$, is equal to the minimum $d$ such that
$P\embeds\R^d$ with the product order.
It follows from this definition that dimension is
subadditive and monotone;
i.e., for any two posets $P$ and $Q$,
$\pdim{P \times Q} \le \pdim{P} + \pdim{Q}$,
and, if $P \embeds Q$,
then $\pdim{P} \leq \pdim{Q}$.

An equivalent definition of dimension
can be given in terms of linear extensions.
Given a poset $P = (S,\leq_P)$,
a \emph{linear extension} of $P$ is a linear order $L = (S,\leq_L)$
that extends $P$, that is,
if $a \leq_P b$, then $a \leq_L b$.
For any poset $P$, a \emph{realiser} of $P$
is a set $\mathcal{L}$ of linear extensions of $P$
with the property that,
for every pair $(a,b)\in P^2$ with $a\not\geq b$,
there exists an $L\in\mathcal{L}$ such that $a\leq_L b$.
Then the dimension of $P$ is the minimum cardinality of a realiser of $P$.
It's a simple exercise to show that the two definitions are equivalent.
Note that there is a standard way to identify linear orders on $[n]$
with permutations.
Namely, given $\sigma \in S_n$,
we associate the order $\leq_\sigma$,
where $a \leq_\sigma b$ if $\sigma(a) \leq \sigma(b)$.


\section{Suborders of the hypercube}
\label{sec:suborders}

Our proof strategy for Theorem~\ref{thm:main}
consists of comparing the dimension of $\divn$ with
the dimension of suborders of the hypercube.
In this section,
we review the theory of Dushnik~\cite{dushnik}
that describes the dimension of suborders of the hypercube
with another combinatorial object:
suitable sets of permutations.

We write $\cube{n}{}$ for the $n$-dimensional hypercube,
that is, the subset lattice of $[n]$.
For any set $A\subseteq[n]$,
$\cube{n}{A}$ denotes the suborder of $\cube{n}{}$
consisting of the subsets $X\subseteq[n]$ with $|X|\in A$.
We write $\cube{n}{a,b}$ instead of $\cube{n}{\{a,b\}}$ for simplicity.

The poset of multisets of $[n]$,
ordered by inclusion with multiplicity,
is denoted $\multi{n}{}$.
For any $A\subseteq \N$,
we denote by $\multi{n}{A}$ the suborder of $\multi{n}{}$ of multisets
whose cardinalities with multiplicity are in $A$,
and by $\widetilde{\mathcal{M}}^n_{A}$ the suborder of $\multi{n}{}$
consisting of all finite multisets
whose underlying sets have cardinalities in $A$,
ignoring multiplicity.
Note that all the posets mentioned are finite,
with the exception of $\widetilde{\mathcal{M}}^n_{A}$.
Usually, we take $A = [0,k]$ or $A = \{1,k\}$.

We will now prove a slightly stronger version
of a lemma by Dushnik~\cite{dushnik},
which characterises the dimension of these posets.
To state the result,
we need a few more definitions.

A \emph{pointed $k$-subset of $[n]$} is an ordered pair $(A,a)$
with $a\in A$, $A \subset [n]$ and $|A| = k$.
A set $S$ of permutations of $[n]$ is called \emph{$k$-suitable}
if, for every pointed $k$-subset $(A,a)$ of $[n]$,
there is a $\sigma\in S$ such that
$b \leq_\sigma a$ for every $b \in A$.
We say that such a $\sigma$ \emph{covers}
the pointed set $(A,a)$.

For any pair $1\leq k \leq n \in\N$,
$N(n,k)$ is defined as the minimum cardinality of
a $k$-suitable set of permutations of $[n]$.
It is clear that $N(n,1)=1$
and that $N(n,2)=2$.
We also have $N(n,k)\geq k$,
since each permutation covers only one of the $k$ pointed sets
on a given underlying set.
Additionally, we have $N(n,k) \leq n$,
since $\{ \sigma, \sigma^2, \dotsc, \sigma^n \}$
is $k$-suitable when $\sigma$ is the permutation
given by $\sigma(i) = i+1 \pmod{n}$.
Because every $k$-suitable set with $2 \leq k \leq n$
is also $(k-1)$-suitable
and the restriction of a $k$-suitable set of permutations of $[n]$
with $1\leq k\leq n-1$ to $[n-1]$
is still $k$-suitable,
$N(n,k)$ is monotone increasing in both arguments.
Later, we will provide upper and lower bounds for $N(n,k)$.

\begin{lemma}
\label{lem:dushnik}
For every $n$ and every $k \leq n-1$,
\begin{equation*}
  \pdim{\cube{n}{1,k}}
  = \pdim{\cube{n}{[0,k]}}
  = \pdim{\multi{n}{[0,k]}}
  = \pdim{\widetilde{\mathcal{M}}^n_{[0,k]}}
  = N(n,k+1).
\end{equation*}
\end{lemma}
\begin{proof}
We show this by proving the following sequence of inequalities:
\begin{equation*}
  N(n,k+1)
  \leq \pdim{\cube{n}{1,k}}
  \leq \pdim{\cube{n}{[0,k]}}
  \leq \pdim{\multi{n}{[0,k]}}
  \leq \pdim{\widetilde{\mathcal{M}}^n_{[0,k]}}
  \leq N(n,k+1).
\end{equation*}

Throughout this proof,
for simplicity,
we identify $[n]$ with the set of one-element subsets of $[n]$.
To show that $N(n,k+1) \leq \pdim{\cube{n}{1,k}}$,
observe that every realiser $\mathcal{L}$ of $\cube{n}{1,k}$
induces a $(k+1)$-suitable set of permutations of $[n]$
in the following way.
For every $L \in \mathcal{L}$,
let $\sigma_L$ be the permutation of $[n]$
induced by the restriction of $L$ to $[n]$.
Now, for every pointed $(k+1)$-set
$(\{a_1, \dotsc, a_{k+1}\}, a_{k+1})$,
there is an $L \in \mathcal{L}$ such that
$\{ a_1, \dotsc, a_k \} \leq_L \{a_{k+1}\}$.
By transitivity,
$\{ a_i \} \leq_L \{a_{k+1}\}$
and hence $a_i \leq_{\sigma_L} a_{k+1}$ for all $i \in [k]$,
so $\{\sigma_L : L \in \mathcal{L}\}$ is $(k+1)$-suitable.

The inequalities
$\pdim{\cube{n}{1,k}}
\leq \pdim{\cube{n}{[0,k]}}
\leq \pdim{\multi{n}{[0,k]}}
\leq \pdim{\widetilde{\mathcal{M}}^n_{[0,k]}}$
hold because each poset embeds into the next.
Now to prove that
$\pdim{\widetilde{\mathcal{M}}^n_{[0,k]}} \leq N(n,k+1)$
we just have to show how to transform
a $(k+1)$-suitable set of permutations
into a realiser of $\widetilde{\mathcal{M}}^n_{[0,k]}$
with the same cardinality.

Let $S$ be a $(k+1)$-suitable set of permutations of $[n]$.
For each $\sigma \in S$,
let $L_\sigma$ be the colexicographic order on
$\widetilde{\mathcal{M}}^n_{[0,k]}$
with respect to $\sigma$.
In other words,
if $A$ and $B$ are two distinct finite multisets of numbers in $[n]$
whose underlying sets have cardinality at most $k$
and $x$ is the $\sigma$-greatest element of $A \cup B$
whose multiplicity in $A$ differs from its multiplicity in $B$,
then $A <_{L_\sigma} B$
if $x$ has greater multiplicity in $B$ than in $A$
and $B <_{L_\sigma}A$
if $x$ has greater multiplicity in $A$.

If $A\subset B$, then $A<_{L_\sigma}B$ for every $\sigma\in S$,
so $L_\sigma$ is a linear extension
of the order on $\widetilde{\mathcal{M}}^n_{[0,k]}$.
If $A$ and $B$ are incomparable in $\widetilde{\mathcal{M}}^n_{[0,k]}$,
then there exists an $x \in B$ whose multiplicity in $B$
is greater than its multiplicity in $A$.
Since $S$ is $\ell$-suitable for every $1 \leq \ell \leq k+1$,
we can find a $\sigma\in S$ that covers $(X\cup\{x\},x)$,
where $X$ is the underlying set of $A$.
Hence $A <_{L_\sigma} B$.
Similarly,
there exists a $y \in A$ whose multiplicity in $A$
is greater than its multiplicity in $B$,
so we can find a $\tau \in S$ such that $B<_{L_\tau}A$.
Therefore $\left\{L_\sigma:\sigma\in S\right\}$
is a realiser of $\widetilde{\mathcal{M}}^n_{[0,k]}$.
\end{proof}

The following result by Dushnik~\cite{dushnik}
gives the exact value of $N(n,k)$,
when $k$ is at least $2\sqrt{n}$.
Note that, by Lemma~\ref{lem:dushnik},
we also obtain the exact dimension of $\cube{n}{1,k}$
and related posets.

\begin{theorem}
\label{thm:dushnik}
For any $j$ and $k$
with $2 \leq j \leq \sqrt{n}$ and
$\big\lfloor \frac{n}{j} \big\rfloor + j - 1
\leq k
\leq \big\lfloor \frac{n}{j-1} \big\rfloor + j - 3$,
we have $N(n,k) = n - j + 1$.
In particular, if $2 \sqrt{n} - 1 \leq k < n$,
then $N(n,k) \geq n - \sqrt{n}$.
\hfill\qedsymbol
\end{theorem}

Spencer proved in \cite{spencer} that,
for all fixed $k \geq 3$,
$N(n,k) = \Theta_k(\log\log n)$
as $n$ grows.
However, the implicit constant in this upper bound grows exponentially in $k$.

The following bound,
which was proved in a slightly stronger form
by Füredi and Kahn~\cite{furedikahn},
is more useful when $\log\log n \ll k \ll \sqrt{n}$,
which is the relevant magnitude
for the proof of the upper bound in Theorem~\ref{thm:main}.

\begin{lemma}
\label{lem:randomsuitable}
For all $1 \leq k \leq n$,
$N(n,k) \leq \lceil k^2 \log n \rceil$.
\end{lemma}
\begin{proof}
The proof is probabilistic.
Fix a natural number $s$
and choose $s$ permutations of $[n]$ independently
and uniformly at random.
The probability that a given pointed $k$-subset isn't
covered by any of these permutations
is $(1-1/k)^s < e^{-s/k}$.
Since the total number of pointed $k$-subsets of $[n]$
is $k \binom{n}{k} \leq n^k$,
the expected number of pointed $k$-subsets not covered
is less than $n^k e^{-s/k} \leq 1$
when $s \geq k^2 \log n$,
so there is an instance where every pointed $k$-subset is covered.
Hence, $N(n,k) \leq \lceil k^2 \log n \rceil$.
\end{proof}


\section{The dimension of divisibility on \texorpdfstring{$[n]$}{[n]}}
\label{sec:dimension_divn}

In this section, we provide a proof of Theorem~\ref{thm:main}.
Additionally, we give similar lower and upper bounds
on the dimension of $\divp{S}$
for other interesting subsets $S\subseteq \N$.
The following principle will be useful
to give upper bounds on the dimension
and the $2$-dimension of $\divn$.

\begin{lemma}
\label{lem:decomposition}
Let $P_1, \dotsc, P_k$ be a partition of the primes in $[n]$
and for $i \in [k]$,
let $Q_i$ be the set of numbers in $[n]$
that can be written as
a (possibly empty) product of powers of primes in $P_i$.
Then
\begin{equation*}
  \divp{[n]} \embeds \divp{Q_1} \times \dotsc \times \divp{Q_k}.
\end{equation*}
\end{lemma}
\begin{proof}
As $P_1, \dotsc, P_k$ is a partition of the primes in $[n]$,
any number $a \in [n]$ can be factored uniquely
as $a = q_1 \dotsc q_k$,
where $q_i \in Q_i$.
Thus, the mapping $a  \mapsto (q_1, \dotsc, q_k)$ is well defined
and we claim that this is the poset embedding we need.
Indeed, if $a = q_1 \dotsc q_k$ and $b = r_1 \dotsc r_k$,
with $q_i,r_i \in Q_i$,
then $a | b$ if and only if $q_i | r_i$ for all $i$.
\end{proof}

Denote by $p_k$ the $k^\textrm{th}$ prime number
and by $\pi(x)$ the number of prime numbers less than or equal to $x$.
We only use standard estimates for these functions:
$p_k = (1 + o(1)) k \log k$
and $\pi(x) = (1 + o(1)) \frac{x}{\log x}$.
Now, we have all the ingredients we need to prove Theorem~\ref{thm:main}.

\begin{proof}[Proof of Theorem~\ref{thm:main}]
First we prove the lower bound.
Observe that, if $k$ is an integer such that
every product of at most $2\sqrt{k}$ distinct elements
of $\{p_1,p_2,\dotsc,p_k\}$ is in $[n]$,
then we have an embedding
$\cube{k}{[0,\lfloor 2\sqrt{k}\rfloor]} \embeds \divn$.
The image of the embedding is the set of all products
of at most $2\sqrt{k}$ of the first $k$ primes,
and is contained in $\divn$
if $p_{k - \lfloor2\sqrt{k}\rfloor+1} \dotsc p_{k} \leq n$.
It follows by Theorem~\ref{thm:dushnik}
and Lemma~\ref{lem:dushnik} that
\begin{equation*}
  \pdim{\divn}
  \geq \pdim{\cube{k}{[0,\lfloor 2\sqrt{k} \rfloor ]}}
  \geq k - \sqrt{k}.
\end{equation*}
This condition for this embedding to exist
is satisfied if $p_k^{2\sqrt{k}} \leq n$.
Now take $\alpha = \alpha(n) < 1/16$ to be chosen later,
and let
$k = \big\lfloor \alpha \big(\frac{\log n}{\log\log n}\big)^2\big\rfloor$.
Using the estimate $p_k = k^{1+o(1)}$,
we obtain
\begin{equation*}
  p_k^{2\sqrt{k}} = k^{(2+o(1))\sqrt{k}}
  \leq \left( \frac{\log n}{\log\log n} \right)^{(4\sqrt{\alpha}+o(1))
    \frac{\log n}{\log\log n}}
  \ll \left(\frac{\log n}{\log\log n}\right)^{\frac{\log n}{\log\log n}}
  < n,
\end{equation*}
whenever $n$ is sufficiently large.
Letting $\alpha$ approach $1/16$ from below,
we obtain
\begin{equation*}
  \pdim{\divn} \geq (1 - o(1))k
  = \big( \tfrac{1}{16} - o(1) \big) \frac{(\log n)^2}{(\log \log n)^2}.
\end{equation*}

To prove the upper bound,
let $\varepsilon = \varepsilon(n) > 0$,
to be chosen later.
Let $S$ be the set of all elements of $[n]$ that can be factored
into primes less than or equal to $(\varepsilon \log n)^2$
and let $R$ be the set of all elements whose prime factors
are all at least $(\varepsilon \log n)^2$.
By Lemma~\ref{lem:decomposition},
we have an embedding
$\divn \embeds \divp{S} \times \divp{R}$,
so $\pdim{\divn} \leq \pdim{\divp{S}} + \pdim{\divp{R}}$.
The poset $\divp{S}$ can then be embedded in
the product of $\pi((\varepsilon\log n)^2)$ chains
(namely the powers of $p$ for each small prime $p$).
Using the weak estimate of the prime number theorem,
$\pi(x) \leq \frac{2x}{\log x}$,
\begin{equation}
\label{eq:ub_smooth}
  \pdim{\divp{S}}
  \leq \pi((\varepsilon\log n)^2)
  \leq \varepsilon^2 \frac{(\log n)^2}{\log\log n + \log \varepsilon}.
\end{equation}
We further partition the large primes.
Let $L = \big\lfloor \log_2\big(\frac{\log n}
{\log\log n+\log\varepsilon}\big) \big\rfloor$
and, for each $0 \leq i < L$,
let $\theta_i = n^{2^{-i}}$.
Thus $\theta_0 = n$,
$\theta_1 = \sqrt{n}$,
and $\varepsilon\log n \leq \theta_L < (\varepsilon\log n)^2$.
Let $R_i$ be the set of numbers in $[n]$
whose prime factors all lie in the interval
$\big(\theta_{i+1}, \theta_i\big]$.
Lemma~\ref{lem:decomposition} now implies that
$\pdim{\divp{R}} \leq \sum_{i=0}^{L-1} \pdim{\divp{R_i}}$.
For every $i$,
the prime factors of each element of $R_i$ forms
a multiset of elements of $\big[\lfloor \theta_i \rfloor \big]$
of cardinality strictly less than $2^{i+1}$.
For all $a,b \in \N$,
$a$ divides $b$ if and only if
the multiset of prime factors of $a$
is a submultiset of the multiset of prime factors of $b$,
so $\divp{R_i} \embeds \multi{\lfloor \theta_i\rfloor}{[0,2^{i+1}-1]}$.
By Lemma~\ref{lem:dushnik}
and Lemma~\ref{lem:randomsuitable},
we have
\begin{equation*}
  \pdim{\divp{R_i}}
  \leq \pdim{ \multi{\lfloor \theta_i \rfloor}{[0,2^{i+1}-1]}}
  = N\big(\lfloor n^{2^{-i}}\rfloor,2^{i+1}\big)
  \leq 4\cdot 2^i\log n + 1.
\end{equation*}
Therefore, we observe that
\begin{align*}
  \pdim{\divp{R}}
  &\leq \sum_{i=0}^{L-1}\pdim{\divp{R_i}}
    \leq \sum_{i=0}^{L-1}4 \cdot 2^i\log n + L \\
  &\leq 4 \cdot 2^L\log n + L
    \leq \frac{4(\log n)^2}{\log\log n + \log\varepsilon} + L.
\end{align*}

Now we choose $\varepsilon(n)$
so that $\varepsilon \to 0$
and $|\log\varepsilon\,| = o(\log\log n)$
as $n \to \infty$.
It suffices to take $\varepsilon = 1/\log\log n$.
Thus we get

\begin{align}
\label{eq:ub_rough}
  \pdim{\divp{R}}
  &\leq (4 - \tfrac{8\log\varepsilon}{\log\log n})
    \frac{(\log n)^2}{\log\log n} \nonumber
    + \log_2\log n - \frac{2\log_2\varepsilon}{\log\log n} \\
  &= (4-o(1))\frac{(\log n)^2}{\log\log n}.
\end{align}

Finally, we combine inequalities (\ref{eq:ub_smooth})
and (\ref{eq:ub_rough}) to obtain
\begin{equation*}
  \pdim{\divn}
  \leq \pdim{\divp{S}} + \pdim{\divp{R}}
  \leq \big( \tfrac{\varepsilon^2}{2} + 4 + o(1) \big)
    \frac{(\log n)^2}{\log\log n},
\end{equation*}
and the result follows immediately.
\end{proof}

Note that the same bounds hold for $\divp{[2,n]}$,
since in the lower bound,
we avoided the element 1 completely.
With some other modifications,
we can adapt our proof to other settings.
We begin by looking at $(n^\alpha, n]$.

\begin{corollary}
\label{cor:ntothealpha}
For any fixed $\alpha \in (0,1)$,
as $n \to \infty$,
\begin{equation*}
  \big( \tfrac{(1 - \alpha)^2}{16} - o_\alpha(1) \big)
    \frac{(\log n)^2}{(\log\log n)^2}
  \leq \pdim{\divp{ (n^\alpha, n]}}
  \leq \big( 4 + o(1) \big)
      \frac{(\log n)^2}{\log \log n}.
\end{equation*}
\end{corollary}
\begin{proof}
These bounds follow from the fact that
$\divp{[2,\lfloor n^{1-\alpha}\rfloor]}
\embeds \divp{(n^\alpha,n]} \embeds \divn$.
The first map
$\divp{[2,\lfloor n^{1-\alpha} \rfloor]}
\embeds \divp{(n^\alpha, n]}$,
defined as $x \mapsto \lfloor n^\alpha \rfloor x$,
is an embedding if $\alpha n > 2$
and the second is just the inclusion map.
\end{proof}

Also of interest is the arithmetic progression
$a[n] + b = \{a k + b : k \in [n] \}$.
If $a$ and $b$ are coprime,
then denote by $p_{a,b,m}$ the $m$-th prime congruent to $b \pmod{a}$.
The prime number theorem for arithmetic progressions implies that
$p_{a,b,n} \sim \varphi(a) n \log n$
as $n \to \infty$,
where $\varphi(a)$ is the Euler totient function,
defined as the order of the multiplicative group $(\Z/a\Z)^{\times}$.

\begin{corollary}
\label{cor:arith}
For any fixed $a$ and $b$,
as $n\to\infty$,
\begin{equation*}
  \big( \tfrac{1}{16} - o_{a,b}(1) \big)
    \frac{(\log n)^2}{(\log \log n)^2}
  \leq \pdim{\divp{a[n] + b}}
  \leq \big( 4 + o_{a,b}(1) \big)
    \frac{(\log n)^2}{\log \log n}.
\end{equation*}
\end{corollary}
\begin{proof}
Since the divisibility poset is dilation-invariant,
we may assume $a$ and $b$ are coprime.
Since $\divp{a[n]+b} \embeds \divp{[an + b]}$,
the upper bound from Theorem~\ref{thm:main} holds.
For the lower bound,
note that $b^{\ell} \equiv b \pmod{a}$
whenever $\ell \equiv 1 \pmod{\varphi(a)}$,
since the order of an element of $(\Z/a\Z)^{\times}$ divides $\varphi(a)$.
Now let
$\ell = \varphi(a)\big\lceil \frac{2\sqrt{k}-1}{\varphi(a)} \big\rceil + 1$,
that is, $2\sqrt{k}$ rounded up to the nearest integer
congruent to $1$ modulo $\varphi(a)$.

We consider now the following map
$f \from \cube{k}{1, \ell} \embeds \divp{a[n] + b}$
defined as follows.
Let $p_{a,b,1}, \dotsc, p_{a,b,k}$ be the first $k$ primes
congruent to $a \pmod{b}$,
and map a set $A \subset \cube{k}{1, \ell}$
to $f(A) = \prod_{i \in A} p_{a,b,i}$.
Note that as $|A| \in \{1, \ell\}$,
we have $f(A) \equiv b^{|A|} \equiv b \pmod{a}$.
It is clear to see that this is indeed an embedding,
given that $k$ is not so large.

Indeed, it is sufficient that $p_{a,b,k}^{\ell} \leq ab + n$.
By the prime number theorem for arithmetic progressions,
we have $p_{a,b,k} \sim \varphi(a) k \log k = k^{1 + o_{a,b}(1)}$,
so $p_{a,b,k}^{\ell} = k^{(2 + o_{a,b}(1))\sqrt{k}}$.
Therefore, a lower bound of the same form
as in Theorem~\ref{thm:main} holds asymptotically
for $\pdim{\divp{a[n]+b}}$.
\end{proof}


\section{The 2-dimension of divisibility on \texorpdfstring{$[n]$}{[n]}}
\label{sec:2dimension_divn}

The goal of this section is to prove Theorem~\ref{thm:main2}.
We first give a formal definition of $t$-dimension.

For any $t\geq 2$,
the \emph{$t$-dimension} of a poset $P$,
denoted $\dim_t\big(P\big)$
is equal to the minimum $d$
such that $P$ can be embedded into a product of $d$ linear orders,
each of cardinality at most $t$.
In particular,
$\twodim{P}$ is the dimension of the smallest hypercube
into which $P$ can be embedded.
As with the Dushnik-Miller dimension,
$t$-dimension is subadditive and monotone for all $t$.

For any poset $P$ with $n$ elements $\{p_1,p_2,\dots,p_n\}$,
the map $p \mapsto \{ i \in [n] : p_i \leq p\}$
is a poset embedding $P \embeds \cube{n}{}$,
so $\twodim{P}\le n$.
Since $\dim_t\big(P\big)$ is monotone decreasing in $t$,
this implies that $\dim_t \big(P \big)$ is well-defined
for every $t \geq 2$ and every finite poset $P$.
We also have the trivial lower bound
$\dim_t \big( P \big) \geq \log_t(|P|)$.
Another useful observation is that a chain of size $\ell$
has $2$-dimension $\ell-1$.

For any $2 \leq k \leq n$,
$N_2(n,k)$ is defined as the minimum cardinality of a set $S$
of subsets of $[n]$ such that,
for any pointed $k$-subset $(A,a)$ of $[n]$,
there exists a set $B \in S$
such that $A \cap B = \{a\}$.
By analogy with $N(n,k)$,
we call such a set a \emph{$k$-suitable family} of subsets.
The following partial analogue to Lemma~\ref{lem:dushnik}
is essentially due to Kierstead~\cite{kierstead}.

\begin{lemma}
\label{thm:2dim_dushnik}
For every $n$ and every $k \leq n + 1$,
\begin{equation*}
  \twodim{\cube{n}{1,k}} = \twodim{\cube{n}{[0,k]}} = N_2(n,k+1)
\end{equation*}
\end{lemma}
\begin{proof}
To show this, we prove the following sequence of inequalities:
\begin{equation*}
  N_2(n, k+1)
    \leq \twodim{\cube{n}{1,k}}
    \leq \twodim{\cube{n}{[0,k]}}
    \leq N_2(n, k+1).
\end{equation*}
To show that $N_2(n,k+1) \leq \twodim{\cube{n}{1,k}}$,
let $d = \twodim{\cube{n}{1,k}}$
and $f \from \cube{n}{1,k} \embeds \cube{d}{}$ be an embedding.
For each $i \in [d]$,
let $X_i$ be the set of all $j\in[n]$
such that $i\in f\big( \{j\} \big)$.
We claim that $\{ X_i : i \in [d] \}$ is $(k + 1)$-suitable.
Indeed, let $(A,a)$ be a pointed $(k+1)$-subset of $[n]$.
Since $f\big( \{a\} \big) \not\subseteq f\big( A\setminus\{a\} \big)$,
there is an $i \in [d]$ such that $i\in f\big( \{a\} \big)$
but $i \not\in f\big( A\setminus\{a\} \big)$.
It follows that $i \not\in f\big( \{b\} \big)$
for any $b\in A\setminus\{a\}$,
so $X_i \cap A = \{a\}$.

The second inequality follows by monotonicity
from the fact that $\cube{n}{1,k} \embeds \cube{n}{[0,k]}$.

To show that $\twodim{\cube{n}{[0,k]}} \leq N_2(n, k+1)$,
let $\{X_1, X_2, \dots, X_d\}$
be a $(k + 1)$-suitable family of subsets of $[n]$.
Define a map $f \from \cube{n}{[0,k]} \to \cube{d}{}$,
$f(A) = \{i \in [d] : A \cap X_i \neq \varnothing \}$.
We claim that this map is an embedding.
Indeed, if $A \subseteq B \subseteq [n]$,
then $f(A) \subseteq f(B)$.
Now, let $A \not\subseteq B$,
where $a \in A$, but $a \notin B$.
Since the family $\{X_i\}$ is $(k+1)$-suitable,
there is $i$ with $(B \cup \{a\}) \cap X_i = \{a\}$,
therefore $X_i \cap B = \varnothing$,
so $i \notin f(B)$,
whereas $i \in f(A)$.
In other words,
$A \not\subseteq B$ implies $f(A) \not\subseteq f(B)$.
Thus $f$ is an embedding
and $\twodim{\cube{n}{[0,k]}} \leq d$.
\end{proof}

An analogue of Lemma~\ref{lem:randomsuitable}
can be proved via the first moment method
by taking random subsets of $[n]$
with each element having probability $\frac{1}{k}$ of being chosen.
This leads to the following theorem,
also by Kierstead~\cite{kierstead}.

\begin{theorem}
\label{thm:kierstead}
For all $2 \leq k \leq n$,
$N_2(n,k) \leq \lceil ek^2\log n \rceil$.
\end{theorem}
\begin{proof}
The proof is probabilistic.
Fix a natural number $s$
and choose a family of $s$ sets of $[n]$ independently as follows.
Each set is select by choosing each element of $[n]$ independently
with probability $1/k$.
The probability that a given pointed $k$-subset
isn't covered by any of the sets in the family is
$\big(1 - (1/k)(1-1/k)^{k-1}\big)^s < e^{-es/k}$.
Since the total number of pointed $k$-subsets of $[n]$
is $k \binom{n}{k} \leq n^k$,
the expected number of pointed $k$-subsets not covered
is less than $n^k e^{-es/k} \leq 1$
when $s \geq e k^2 \log n$,
so there is an instance where every pointed $k$-subset is covered.
Hence, $N(n,k) \leq \lceil e k^2 \log n \rceil$.
\end{proof}

Because the $2$-dimension of a poset depends in part on its cardinality,
the full analogue of Lemma~\ref{lem:dushnik} is false.
We need the following lemma to extend the result
we just obtained to multisets
in order to give the upper bound in Theorem~\ref{thm:main2}.

\begin{lemma}
\label{lem:multi}
For all $k \leq n-1$,
$\twodim{\multi{n}{[0,k]}} < e(\frac{\pi^2}{6}k^2+2k\log k+3k)\log n + k$.
\end{lemma}
\begin{proof}
For each $A \in \multi{n}{[0,k]}$
and each $i \in [k]$,
let $A^i$ be the set of all elements of $A$ of multiplicity at least $i$.
Observe that $i|A^i| \leq |A| \leq k$,
so $|A^i| \leq k/i$.
For any two multisets $A$ and $B$ in $\multi{n}{[0,k]}$,
$A\subseteq B$ if and only if
$A^i\subseteq B^i$ for every $i \in [k]$.
Hence the map $A \mapsto (A^1,\dotsc,A^k)$
is a poset embedding
$\multi{n}{[0,k]} \embeds \prod_{i=1}^k \cube{n}{[0,\lfloor k/i\rfloor]}$.
Therefore, by Lemma~\ref{thm:2dim_dushnik}
and Theorem~\ref{thm:kierstead},
we have
\begin{align*}
  \twodim{\multi{n}{[0,k]}}
    &\leq \sum_{i=1}^k \twodim{\cube{n}{[0,\lfloor k/i\rfloor]}}
      \leq \sum_{i=1}^k N_2(n,\lfloor k/i\rfloor+1) \\
    &< e \log n \sum_{i=1}^k \Big( \frac{k^2}{i^2}+\frac{2k}{i}+1 \Big) + k \\
    &\leq e \big(\tfrac{\pi^2}{6}k^2 +2k(\log k+1) + k \big)\log n + k.
    \qedhere
\end{align*}
\end{proof}

Since $\pdim{\divn} \leq \twodim{\divn}$,
Theorem~\ref{thm:main} already provides a lower bound for $\twodim{\divn}$.
Therefore, to prove Theorem~\ref{thm:main2},
only the proof of the upper bound is required.

\begin{proof}[Proof of Theorem~\ref{thm:main2}]
The proof is essentially the same as that
of the upper bound of Theorem~\ref{thm:main},
so we will omit some of the details.
Take $\varepsilon = \varepsilon(n) > 0$,
to be chosen later.
Let $S$ be the set of all elements of $[n]$
whose prime factors are all at most $\varepsilon\log n$
and $R$ be the set of all elements whose prime factors
are all greater than $\varepsilon \log n$.

The poset $\divp{S}$ can be embedded
into the product of $\pi(\varepsilon\log n)$ chains,
each of length at most $1 + \log_2 n$.
Since the $2$-dimension of a chain of length $\ell$ is $\ell - 1$,
we have
\begin{equation*}
  \twodim{\divp{S}}
    \leq \pi(\varepsilon \log n)\log_2 n
    = \frac{2\varepsilon}{\log 2}
      \cdot \frac{(\log n)^2}{\log\log n + \log \varepsilon}.
\end{equation*}

Let
$L = \big\lceil
  \log_2\big(\frac{\log n}{\log\log n + \log \varepsilon}\big)
\big\rceil$.
For each $i$ from $0$ to $L$,
let $\theta_i = n^{2^{-i}}$
(so that $\sqrt{\varepsilon\log n} < \theta_L \leq \varepsilon\log n$)
and $R_i$ be the set of elements of $[n]$
whose prime factors all lie in the interval
$\big(\theta_{i+1},\theta_i\big]$.
Just as before,
we have embeddings
$\divp{R} \embeds \prod_{i=0}^{L-1} \divp{R_i}$
and $\divp{R_i} \embeds \multi{\lfloor \theta_i \rfloor}{[0,2^{i+1}-1]}$.
By Lemma~\ref{lem:multi},
we have
\begin{align*}
  \twodim{\divp{R_i}}
  \leq \twodim{\multi{\lfloor \theta_i \rfloor}{[0,2^{i+1}-1]}}
    &\leq \Big(\frac{2e\pi^2}{3}\cdot 2^i + 4e\log 2 \cdot (i+1) + 6e\Big)\log n
      + 2^{i+1} \\
    &\leq \Big(\frac{2e\pi^2}{3}\cdot 2^i + 8i + 24\Big)\log n + 2^{i+1}.
\end{align*}

Therefore, we obtain the following bound:
\begin{equation*}
  \twodim{\divp{R}}
    \leq \sum_{i=0}^{L-1} \twodim{\divp{R_i}}
    \leq \frac{2e\pi^2}{3}\cdot 2^L\log n + 4 L^2 \log n + 24 L \log n
        + 2^{L+1}.
\end{equation*}

As before,
we need $\varepsilon$ to go to $0$ slowly;
it suffices to take $\varepsilon = 1/\log\log n$.
We then have
\begin{equation*}
  \twodim{\divp{R}}
    \leq \Big(\frac{4e\pi^2}{3} + o(1)\Big)\frac{(\log n)^2}{\log\log n}.
\end{equation*}

Note that the extra factor of 2 in the constant
comes from the ceiling in the definition of $L$.
Finally, we have
$\twodim{\divn}
  \leq \twodim{\divp{S}} + \twodim{\divp{R}}
  \leq \big(\frac{2\varepsilon}{\log 2} + \tfrac{4e\pi^2}{3}+o(1)\big)
    \frac{(\log n)^2}{\log\log n}$,
and, because $\varepsilon \to 0$,
the stated upper bound follows.
\end{proof}

We note that the analogues of Corollaries~\ref{cor:ntothealpha}
and \ref{cor:arith} hold for $2$-dimension as well.

\begin{corollary}
\label{cor:ntothealpha2dim}
\label{cor:arith2dim}
For any fixed $\alpha \in (0,1)$
and $a, b \in \N$,
as $n\to\infty$,
\begin{equation*}
  \big( \tfrac{(1 - \alpha)^2}{16} - o_\alpha(1) \big)
    \frac{(\log n)^2}{(\log\log n)^2}
  \leq \twodim{\divp{(n^\alpha, n]}}
  \leq \big( \tfrac{4}{3}e\pi^2 + o(1) \big)
    \frac{(\log n)^2}{\log \log n},
\end{equation*}
\begin{equation*}
\pushQED{\qed}
  \big( \tfrac{1}{16} - o_{a,b}(1) \big)
    \frac{(\log n)^2}{(\log \log n)^2}
  \leq \twodim{\divp{a[n] + b}}
  \leq \big( \tfrac{4}{3}e\pi^2 + o_{a,b}(1) \big)
  \frac{(\log n)^2}{\log \log n}. \qedhere
\popQED
\end{equation*}
\end{corollary}

The proofs are nearly identical to the ones for dimension,
so we omit them.


\section{The dimension of the divisibility order on \texorpdfstring{$(\alpha n, n]$}{(an,n]}}
\label{sec:alpha}

In previous sections,
we have already considered the dimension of the divisibility order
on sets other than $[n]$,
such as $(n^{\alpha}, n]$ or $a[n] + b$.
The proof of Theorem~\ref{thm:main}
can be adapted to those cases after some small modifications.
In this section,
we will study the divisibility order on $(\alpha n, n]$,
whose dimension behaves in a different manner.
Indeed,
$\divp{(\alpha n, n]}$ is an antichain when $\alpha > 1/2$,
for instance,
so it has dimension only $2$.

The \emph{comparability graph} of a poset $P$
is the graph with vertex set $P$
where two elements are connected
if they are comparable in $P$.
A theorem by F\"{u}redi and Kahn~\cite{furedikahn}
states that a poset whose comparability graph
has maximum degree $\Delta$
has dimension less than $50\Delta(\log \Delta)^2$.
This bound was recently improved by Scott and Wood \cite{scottwood},
who showed that the maximum dimension
of a poset of maximum degree $\Delta$
is $\Delta (\log \Delta)^{1+o(1)}$
as $\Delta \to \infty$.

The comparability graph of $\divp{(\alpha n, n]}$
has maximum degree at most $1/\alpha + 1$.
Indeed, let $x \in (\alpha n, n]$
with $x = \beta n$
for some $\beta \in (\alpha, 1]$.
The number of elements from $(\alpha n, n]$ that divide $x$
is at most $\beta / \alpha$
and the number of those that are divisible by $x$
is at most $1 / \beta$,
so the degree of $x$ in the comparability graph
is at most $1/\beta + \beta/\alpha \leq 1 + 1/\alpha$.
Therefore, as $\alpha \to 0$, we have
\begin{equation*}
  \sup\limits_{n\in\N} \pdim{\divp{(\alpha n, n]}}
  \leq \tfrac{1}{\alpha} \big( \log(\tfrac{1}{\alpha}) \big)^{1+o(1)} .
\end{equation*}

We note that the $t$-dimension of $\divp{(\alpha n, n]}$
has a very distinct behaviour from the ordinary dimension,
since these posets have unbounded cardinality
and hence unbounded $t$-dimension.

For a poset $P$ and $x\in P$,
we define the outdegree of $x$ as
$\big| \{y \in P : y > x \}\big|$
and the indegree of $x$ as
$\big| \{y \in P : y < x \}\big|$.
Another theorem by F\"{u}redi and Kahn \cite{furedikahn}
says that a poset of cardinality $n$
and maximum outdegree $\upsilon$
has dimension at most
$\lceil 2(\upsilon+2)\log n \rceil$.
The following lemma gives similar bounds for $2$-dimension.

\begin{lemma}
\label{lem:furedikahn2dim}
Let $P$ be a poset of cardinality $n$,
maximum outdegree $\upsilon$,
and maximum indegree $\delta$.
Then we have the following bounds:
\begin{equation}
\label{eq:fk2d1}
    \twodim{P}
    \leq \lceil 2e (\upsilon + 2) \log n \rceil
\end{equation}
\begin{equation}
\label{eq:fk2d2}
    \twodim{P}
    \leq \big\lceil e(\upsilon+2)\big(\log n + \log(\upsilon+2)
      + \log(\delta+2) + 1\big) \big\rceil.
\end{equation}
\end{lemma}
\begin{proof}
Let $P$ be a poset of cardinality $n$
and maximum outdegree $\upsilon$.
We are going to construct an embedding
from $P$ into $\cube{d}{}$ randomly,
for $d$ sufficiently large.
For each $x \in P$,
let $A_x$ be an independent random subset of $[d]$,
where each element is selected independently with probability
$p = 1 - \frac{1}{\upsilon + 2}$.
We define a map $f \from P \to \cube{d}{}$,
$f(x) = \bigcap_{y \geq x} A_y$.
Our goal is to show that,
if $d$ is large enough,
then with positive probability $f$ is a poset embedding.

Note that $f$ is monotone by construction.
It is an embedding if and only if,
for every pair $(x,y) \in P^2$
with $x \not\leq y$,
we have $f(x) \not\subseteq f(y)$.
For each such pair $(x,y)$,
let $E_{x,y}$ be the event that $f(x) \subseteq A_y$.
Since $f(y) \subseteq A_y$,
if none of the events $E_{x,y}$ occurs,
then $f$ is a poset embedding.
For each $i \in [d]$,
we have $\Prob ( i\in f(x) , i \not\in A_y ) \geq p^{\upsilon + 1}(1-p)$,
so
\begin{equation*}
  \Prob\left(E_{x,y}\right)
    \leq (1 - p^{\upsilon + 1}(1-p))^d
    \leq \exp \big( -p^{\upsilon + 1}(1-p)d \big)
    \leq \exp \big(-\tfrac{d}{e(\upsilon+2)} \big).
\end{equation*}

To prove (\ref{eq:fk2d1}),
choose $d \geq 2e(\upsilon + 2)\log n $.
The expected number of events $E_{x,y}$ that occur
is at most $(n^2-n)n^{-2}<1$,
so with positive probability none of them occurs.

To prove (\ref{eq:fk2d2}),
we use the following form of the Lovász local lemma:

\begin{lemma}[Lovász Local Lemma, Theorem 1.5 in \cite{lll}]
Suppose $0 < p < 1$
and let $A_1, A_2, \dots, A_k$ be events in a probability space
such that $\Prob(A_i) \leq p$.
Let $G$ be a graph with vertex set $[k]$ such that,
for all $i \neq j \in [k]$,
$A_i$ and $A_j$ are independent unless $ij \in E(G)$,
and suppose $G$ has maximum degree $\Delta$.
If $ep(\Delta+1)\leq 1$,
then $\Prob \big( \bigcap\limits_{i=1}^k \overline{A_i} \big) > 0$.
\hfill\qedsymbol
\end{lemma}

The event $E_{x,y}$ is independent from $E_{z,w}$
if the sets $\{y\} \cup \{u : u\geq x\}$
and $\{w\} \cup \{u : u\geq z\}$ are disjoint.
If they are not disjoint,
then either $w = y$,
or $z \leq y$,
or $w \geq x$,
or $x$ and $z$ have a common upper bound.
For fixed $x$ and $y$,
the number of choices for $(z,w)$ such that these sets intersect
(not counting $(x,y)$ itself)
is therefore at most
$n + (\delta+1)n + (\upsilon+1)n + (\upsilon+1)(\delta+1)n - 1
  = (\upsilon+2)(\delta+2)n-1$.

Hence the total number of events $E_{z,w}$ dependent on $E_{x,y}$
is at most $(\upsilon+2)(\delta+2) n-1$.
If we choose
$d \geq e(\upsilon+2)(\log n + \log(\upsilon+2) + \log(\delta+2) + 1)$,
then $e(\upsilon+2)(\delta+2) n e^{-\frac{d}{e(\upsilon+2)}} \leq 1$,
and by the Lovász Local Lemma,
the probability that none of the events $E_{x,y}$ occurs is positive.
\end{proof}


Using this result,
we can bound the $t$-dimension of $\divp{(\alpha n, n]}$
for any fixed $t$ and $\alpha$.
This poset has at least $(1-\alpha)n - 1$ elements,
so its $t$-dimension is at least
\begin{equation}
\label{eq:lb_tdim_alphan}
  \log_t\big(  (1 - \alpha) n - 1\big) = \frac{\log n}{\log t} - O_{\alpha,t}(1).
\end{equation}
We can apply Lemma~\ref{lem:furedikahn2dim} to obtain an upper bound.
As we have seen before,
the maximum indegree and outdregree of $\divp{(\alpha n,n]}$
is at most $1/\alpha$,
and its cardinality is at most $(1-\alpha)n$,
so by (\ref{eq:fk2d2})
\begin{equation*}
  \dim_t\big( \divp{(\alpha n,n]} \big)
    \leq \twodim{\divp{(\alpha n,n]}}
    \leq (e+o_\alpha(1)) \big( \tfrac{1}{\alpha} \big) \log n.
\end{equation*}
This,
together with the lower bound in (\ref{eq:lb_tdim_alphan}),
implies that
$\dim_t \big(\divp{(\alpha n,n]}\big) = \Theta_{\alpha,t}(\log n)$
as $n \to \infty$.


\section{Open questions}
\label{sec:open}

We pose several problems in this section,
of which the central one is the following.

\begin{question}
What is the correct asymptotic order of growth of $\pdim{\divn}$?
\end{question}

We do not make any prediction of whether the lower bound
or the upper bound in Theorem~\ref{thm:main}
is closer to the truth.
On one hand,
the lower bound is sharp in the sense that
no $\cube{k}{1,\ell}$ of higher dimension
can be embedded into $\divn$,
but on the other,
the upper bound is more technically refined,
where we bound each layer appropriately.
In any case,
we believe that determining the correct exponent
on the $\log\log n$ factor requires new ideas.
But we conjecture that for $\divn$,
dimension and $2$-dimension should behave similarly.

\begin{conjecture}
$\twodim{\divn} = \Theta\big( \pdim{\divn} \big)$
as $n\to\infty$.
\end{conjecture}

So far,
we have seen how the dimension behaves for some specific well structured sets,
like $[n]$ and $a[n] + b$.
How does the dimension of a typical set behave?

\begin{problem}
Let $p = p(n)$
and let $A \subseteq [n]$ be a random subset
where each element is chosen independently with probability $p$.
How does $\pdim{\divp{A}}$ grow with $n$?
\end{problem}

Although we believe this question to be of great interest,
we have made no serious attempt to answer it.
It would be interesting to see how other poset properties vary with $p$.

We have shown in Section~\ref{sec:alpha}
that the dimension of $\divp{(\alpha n, n]}$ is bounded for all $n$.
Indeed, we have shown an upper bound of
$\tfrac{1}{\alpha} \big(\log(\tfrac{1}{\alpha})\big)^{1+o(1)}$
as $\alpha \to 0$.
A lower bound of
$(\tfrac{1}{16} - o(1))
  \big(\log(\tfrac{1}{\alpha}) / \log\log(\tfrac{1}{\alpha})\big)^2$
for $\alpha$ sufficiently small
can be obtained by embedding $\divp{[2,\lfloor 1/\alpha \rfloor]}$
into $\divp{(\alpha n, n]}$
by multiplying every element by $\lfloor \alpha n \rfloor$,
noting again that in the lower bound of $\pdim{\divn}$
does not use the element 1.
It would be nice to improve the bounds obtained.
We also believe that
$\lim_{n \to \infty} \pdim{\divp{(\alpha n, n]}}$
exists for all $\alpha$.

\begin{problem}
How does $\sup_{n\in\N} \pdim{\divp{(\alpha n, n]}}$
increase as $\alpha \to 0$?
\end{problem}

Finally,
recall that we have shown that
$\dim_t(\divp{(\alpha n, n]}) = \Theta_{\alpha, t}\big( \log n\big)$.
This suggests the following conjecture.

\begin{conjecture}
For each $0 < \alpha < 1$
and $t \geq 2$,
there exists a constant $c = c(\alpha,t)$
such that $\dim_t \big( \divp{(\alpha n,n]} \big) \sim c\log n$
as $n \to \infty$.
\end{conjecture}


\section{Acknowledgements}

The first author was supported by
NSF-DMS grants \#1855745,
``Applications of Probabilistic Combinatorial Methods''
and \#1600742,
``Probabilistic and Extremal Combinatorics''.
The second author was partially supported by
CAPES, Brazil.

We thank the referees for their
careful reading of the initial manuscript
and for their very helpful comments.



\end{document}